\documentclass[12pt]{amsart}
\usepackage[margin=3.5cm]{geometry}
\usepackage{amscd}
\usepackage{amsfonts,amssymb,amsthm}
\usepackage{graphics}

\theoremstyle{plain}
\newtheorem{thm}{Theorem}[section]
\newtheorem{lem}[thm]{Lemma}
\newtheorem{cor}[thm]{Corollary}
\newtheorem{prop}[thm]{Proposition}

\theoremstyle{definition}

\theoremstyle{remark}

\newcommand{\script}{\mathcal}
\newcommand{\bbf}{\mathbb}

\newcommand{\bksl}{\setminus}

\newcommand{\vepsilon}{\epsilon}

\DeclareMathOperator{\stab}{stab}

\DeclareMathOperator{\inj}{\inj}

\DeclareMathOperator{\Map}{Map}

\DeclareMathOperator{\Aut}{Aut}
\DeclareMathOperator{\Out}{Out}
\DeclareMathOperator{\Inn}{Inn}

\newcommand{\ZZ}{\mathbb{Z}}
\newcommand{\HH}{\mathbb{H}}
\newcommand{\ray}{\mathcal{R}}

\newcommand{\pph}{\mathcal{P}}

\newcommand{\qq}{\mathbf{q}}
\newcommand{\gs}{\mathbf{g}}

\begin{document}

\title[The space of geodesic rays]
{The action of the mapping class group
on the space of geodesic rays of a punctured hyperbolic surface}

\author{Brian H. Bowditch}
\address{Mathematics Institute\\
University of Warwick\\
Coventry CV4 7AL, Great Britain}

\author{Makoto Sakuma}
\address{Department of Mathematics\\
Graduate School of Science\\
Hiroshima University\\
Higashi-Hiroshima, 739-8526, Japan}
\email{sakuma@math.sci.hiroshima-u.ac.jp}

\date{9th August 2016}

\subjclass[2010]{20F65, 20F34, 20F67, 30F60, 57M50 \\
\indent {
The second author was supported by JSPS Grants-in-Aid 15H03620.}}

\begin{abstract}
Let $\Sigma$ be a complete finite-area orientable hyperbolic surface 
with one cusp,
and let $\ray$ be the space of complete geodesic rays 
in $\Sigma$ emanating from the puncture. 
Then there is a natural action of the mapping class group of $\Sigma$ 
on $\ray$.
We show that this action is \lq\lq almost everywhere'' wandering.
\end{abstract}

\maketitle

\section{Introduction}\label{SA}
Let $\Sigma$ be a complete finite-area orientable hyperbolic surface 
with one cusp,
and $\ray$ the space of complete geodesic rays 
in $\Sigma$ emanating from the puncture. 
Then, there is a natural action of the (full) mapping class group
$\Map(\Sigma)$ of $\Sigma$ on $\ray \equiv S^1$ (see Section \ref{SB}).
The dynamics of the action of an element of $\ray$
plays a key role in the Nielsen-Thurston theory for surface homeomorphisms.
It also plays a crucial role in the variation of McShane's identity
for punctured surface bundles with pseudo-Anosov monodromy, 
established by \cite{Bo1} and \cite{AkMS}.

It is natural to ask what does the action of the 
whole group $\Map(\Sigma)$ (or its subgroups) look like.
However, the authors could not find a reference
which treats this natural question,
though there are various references
which study the action of (subgroups of) the mapping class groups
on the projective measured lamination spaces,
which are homeomorphic to higher dimensional spheres
(see for example, \cite{Mas1, Mas2, MccP, OS}).
In particular, such an action is minimal (cf. \cite{FatLP})
and moreover ergodic \cite{Mas1}.

The purpose of this paper is to prove that
the action of $\Map(\Sigma)$ on $\ray$ 
is \lq\lq  almost everywhere'' wandering
(see Theorem \ref{B1} for the precise meaning).
This forms a sharp contrast to the above result of \cite{Mas1}.

We would like to thank Katsuhiko Matsuzaki 
for his helpful comments on the first version of the paper.

\section{Actions}\label{SB}
Let $ \Sigma=\HH^2/\Gamma $ be a complete finite-area orientable hyperbolic surface with precisely one cusp, where $ \Gamma=\pi_1(\Sigma) $.
Let $ \ray $ be the space of complete geodesic rays 
in $ \Sigma $ emanating from the puncture. 
Then $\ray $ is identified with
a horocycle, $\tau$, in the cusp.
In fact, a point of $ \tau $ determines a geodesic ray in $ \Sigma $
emanating from the puncture,
or more precisely, a bi-infinite geodesic path with 
its positive end
going out the cusp and meeting $ \tau $ in the given point.
Any mapping class $\psi$ of $ \Sigma $ maps each geodesic ray to another path which
can be \lq\lq straightened out'' to another geodesic ray, 
and hence determines another
point of $ \tau $.
This gives an action of the infinite cyclic group 
generated by $ \psi $ on $ \ray \equiv \tau $.

A rigorous construction of this action is described as follows.
Choose a representative, $ f $, of $ \psi $,
so that its lift $ \tilde f $ to the universal cover $ \HH^2 $
is a quasi-isometry.
Then $ \tilde f $ extends to a self-homeomorphism 
of the closed disc $ \HH^2\cup \partial\HH^2 $.
For a geodesic ray $ \nu \in \ray $,
let $ \tilde \nu $ be the closure in $ \HH^2\cup \partial\HH^2 $
of a lift of $ \nu $ to $ \HH^2 $.
Then $\tilde f (\tilde\nu) $ is an arc 
properly embedded in $ \HH^2\cup \partial\HH^2 $,
and its endpoints determine a geodesic in $ \HH^2 $,
which project to another geodesic ray $ \nu' \in \ray$.
Thus, we obtain an action of $ \psi $ on $ \ray $,
by setting $\psi\nu=\nu' $.
The dynamics of this action plays a key role in \cite{AkMS}.
However, one needs to verify that this action does not depend on
the choice of a representative $ f $ of $ \psi $.

In the following, we settle this issue, by using 
the canonical boundary of
a relatively hyperbolic group described in \cite{Bo2}.
Though we are really interested here only in the case where
the group is the fundamental group of a once-punctured closed orientable surface, and the the peripheral structure is interpreted in the usual way 
(as the conjugacy class of the fundamental group of
a neighborhood of the puncture),
we give a discussion in a general setting.

Let $ \Gamma $ be a non-elementary relatively hyperbolic group with a
given peripheral structure $ \pph $, 
which is a conjugacy invariant collection of 
infinite subgroups of $ \Gamma $.
By \cite[Definition 1]{Bo2},
$ \Gamma $ admits a properly discontinuous isometric action 
on a path-metric space, $ X $, with the following properties.
\begin{enumerate}
\item
$ X $ is proper (i.e., complete and locally compact) and Gromov hyperbolic, 
\item
every point of the boundary of $ X $ is either a conical limit point or a bounded parabolic point,
\item
the peripheral subgroups, i.e.,
the elements of $ \pph $,
are precisely 
the maximal parabolic subgroups of $ \Gamma $, and
\item
every peripheral subgroup is finitely generated.
\end{enumerate}
It is proved in \cite[Theorem 9.4]{Bo2} that
the Gromov boundary $ \partial X $ is uniquely determined
by $ (\Gamma, \pph) $,
(even though the quasi-isometry class of the space $ X $
satisfying the above conditions is not uniquely determined).
Thus the boundary 
$ \partial \Gamma = \partial(\Gamma,\pph)$ 
is defined to be $ \partial X $.
By identifying $ \Gamma $ with an orbit in $ X $,
we obtain a natural topology on
the disjoint union $ \Gamma \cup \partial \Gamma $ 
which is compact Hausdorff, with
$ \Gamma $ discrete and $ \partial \Gamma $ closed.

The action of $ \Gamma $ on itself by left multiplication extends
to an action on $ \Gamma \cup \partial \Gamma $ by homeomorphism.
This gives us a geometrically finite convergence action of
$ \Gamma $ on $ \partial \Gamma $.
Let $ \Aut(\Gamma,\pph) $ be the subgroup of 
the automorphism group, $ \Aut(\Gamma) $, of $ \Gamma $
which respects the peripheral structure $ \pph $.
This contains the inner automorphism group, $ \Inn(\Gamma) $.
Now, by the naturality of $ \partial\Gamma $
(\cite[Theorem 9.4]{Bo2}),
the action of $ \Aut(\Gamma,\pph) $ on $ \Gamma $ also extends to
an action on 
$ \Gamma \cup \partial \Gamma $,
which is  $ \Gamma $-equivariant,
i.e., $ \phi\cdot(g\cdot x) = \phi(g)\cdot(\phi\cdot x) $
for every $ \phi \in \Aut(\Gamma,\pph) $, $ g \in \Gamma $ and 
$ x \in \Gamma \cup \partial \Gamma $.
(In order to avoid confusion, we use $\cdot$ to denote group actions, 
only in this place.)
Under the natural epimorphism $ \Gamma \longrightarrow \Inn(\Gamma) $,
this gives rise to the same action 
on $\partial\Gamma$
as that induced by left multiplication.
The centre of $ \Gamma $ is always finite, and for simplicity, we
assume it to be trivial.
In this case, we can identify $ \Gamma $ with $ \Inn(\Gamma) $.

Suppose that $ p \in \partial \Gamma $ is a parabolic point.
Its stabiliser, $ Z = Z(\Gamma,p) $, in $ \Gamma $ is a peripheral subgroup.
Now $ Z $ acts properly discontinuously cocompactly on
$ \partial \Gamma \bksl \{ p \} $, so the quotient
$ T = (\partial \Gamma \bksl \{ p \})/Z $ is compact Hausdorff
(cf. \cite[Section 6]{Bo2}).
Let $ A=A(\Gamma,\pph,p) $ be the stabiliser of 
$ p $ in $ \Aut(\Gamma,\pph) $.
Then $ Z $ is a normal subgroup of $ A $, 
and we get an action of $ M = A/Z $ on $ T $.
If there is only one conjugacy class of peripheral subgroups, then the
orbit $ \Gamma p $ is $ \Aut(\Gamma,\pph) $-invariant, and it follows that
the group $ A $ maps 
isomorphically onto
$ \Out(\Gamma,\pph) =\Aut(\Gamma,\pph)/\Inn(\Gamma) $, 
so in this case we
can naturally identify the group $ M $ with $ \Out(\Gamma,\pph) $.

Suppose now that $ \Sigma $ is a once-punctured closed orientable surface,
with negative Euler characteristic $ \chi(\Sigma) $.
We write $ \Sigma = D/\Gamma $, where $ D = {\tilde \Sigma} $, 
the universal cover, and
$ \Gamma \cong \pi_1(\Sigma) $.
Let $ \pph $ be the peripheral structure of $ \Gamma $ arising 
from the cusp of $ \Sigma $,
namely $ \pph $ consists of the conjugacy class of the fundamental group of 
a neighbourhood of the end of $ \Sigma $.
Then $ (\Gamma,\pph) $ is a relatively hyperbolic group,
because if we fix a complete hyperbolic structure on $ \Sigma $
then $ D $ is identified with $ \HH^2 $ and the isometric action of 
$ \Gamma $ on $ D=\HH^2 $ satisfies the conditions (1)--(4) in the above,
namely \cite[Definition 1]{Bo2}.
Now $ D $ admits a natural compactification to a closed disc, $ D \cup C $,
where $ C $ is the dynamically defined circle at infinity.
We can identify $ C $ with $ \partial \Gamma $.
In fact, if $ x $ is any point of $ D $, then identifying $ \Gamma $ with
the orbit $ \Gamma x $, we get an identification of
$ \Gamma \cup \partial \Gamma $ with $ \Gamma x \cup C \subseteq
D \cup C $.
As above we get an action of $ \Aut(\Gamma,\pph) $ on $ C $.
If $ p \in \partial C $ is parabolic, 
then its stabiliser $ Z $ in $ \Gamma $ is
isomorphic to 
the infinite cyclic group
$ \bbf Z $, and we get an action of 
$ \Out(\Gamma,\pph) $
on the circle $ T = (C \bksl \{ p \})/Z $.
Since $ \Out(\Gamma,\pph) $ is identified with 
the (full) mapping class group, $ \Map(\Sigma) $, of $ \Sigma $,
we obtain a well defined action
of $ \Map(\Sigma) $ on the circle $ T $.

We now return to the setting in the beginning of this section,
where $ \Sigma=\HH^2/\Gamma $ is endowed with a complete hyperbolic structure.
Then we can identify the (dynamically defined) circle $T$ with 
the horocycle, $ \tau $, in the cusp,
which in turn is identified with the space of geodesic rays, $ \ray $.
This gives an action of $ \Map(\Sigma) $ on $ \ray $.
Since the action of $ \Gamma $ on $ \HH^2 $
satisfies the conditions (1)-(4) in the above
(i.e., \cite[Definition 1]{Bo2}),
we see that, 
for each mapping class $ \psi $ of $ \Sigma $,
its action on $ \ray $, defined via the \lq\lq straightening process''
presented at the beginning of this section,
is identical with the action which is dynamically constructed in the above,
independently from the hyperbolic structure.
Thus the problem raised at the beginning of this section is settled.

\bigskip

In order to state the main result, 
we prepare some terminology.
Let $ G $ be a group acting by homeomorphism on a topological space $ X $.
An open subset, $ U \subseteq X $, is said to be 
\emph{wandering}
if $ gU \cap U = \emptyset $ for all $ g \in G \bksl \{ 1 \} $.
(Note that this definition is stronger than the usual definition
of wandering, where it is only assumed that
the number of $g\in G$ such that $ gU \cap U \ne \emptyset $ is finite.)
The \emph{wandering domain}, $ W_G(X) \subseteq X $ is the union of
all wandering open sets.
Its complement, $ W^c_G(X) = X \bksl W_G(X) $, is the \emph{non-wandering set}.
This is a closed $ G $-invariant subset of $ X $.
Note that if $ Y \subseteq X $ is a $G$-invariant open set, 
then $ W_G(Y) = W_G(X) \cap Y $.
If $ H \triangleleft G $ is a normal subgroup, we get an induced action
of $ G/H $ on $ X/H $.
(In practice, the action of $ H $ on $ X $ will be properly discontinuous.)
One checks easily that 
$ W_G(X)/H  \subseteq  W_{G/H}(X/H) $
with equality if $ W_H(X) = X $.

Note that any hyperbolic structure on $ \Sigma $ induces a
euclidian metric on $ T $ (via the horocycle $ \tau $).
If one changes the hyperbolic metric, the induced euclidian metrics
on $ T $ are related by a quasisymmetry.
However, they are completely singular with respect to each other
(see \cite{Ku, Tu2}).
(That is, there is a set which has zero measure in one structure,
but full measure in the other.)
In general, this gives little control over how the Hausdorff dimension
of a subset can change.

We say that a subset, $ B \subseteq T $ is \emph{small} if it has
Hausdorff dimension stricty less than 1 with respect to any hyperbolic
structure on $ \Sigma $.
Now we can state our main theorem.

\begin{thm}\label{B1}
Let $ \Sigma $  be
a once-punctured closed orientable surface,
with $ \chi(\Sigma)<0 $,
and consider the action of $ \Map(\Sigma) $ on the circle $ T $,
defined in the above.
Then the non-wandering set in $ T $ 
with respect to the action of $ \Map(\Sigma) $
is small.
\end{thm}

In particular, the non-wandering set has measure 0 with respect
to any hyperbolic structure, and so has empty interior.

Given that two different hyperbolic structures give rise to
quasisymmetically related metrics on $ T $, it is natural
to ask if there is a more natural way to express this.
For example, is there a property of (closed) subsets of $ T $,
invariant under quasisymmetry and satisfied by the non-wandering set,
which implies Hausdorff dimension less than 1 
(or measure 0)?

\section{The loop-cutting construction}\label{SC}

Let $ \Sigma=\HH^2/\Gamma $ be a complete finite-area orientable hyperbolic surface with precisely one cusp, where $ \Gamma=\pi_1(\Sigma) $.
Thus the universal cover $ D = {\tilde \Sigma} $ is identified with 
the hyperbolic plane $ \HH^2 $.
Write $ C $ for the ideal boundary of $ D $, which we consider equipped
with a preferred orientation.
Thus $ \Gamma $ acts on $ C $ as a geometrically finite convergence group.
Let $ \Pi \subseteq C $ be the set of parabolic points of $ \Gamma $.
Given $ p \in \Pi $, let $ \theta(p) $ be the generator of
$ \stab_\Gamma(p) $ which 
acts on $ C\bksl \{p\} $ as a translation in the positive direction.
Given distinct $ x,y \in C $, let $ [x,y] \subseteq D \cup C $ denote 
the oriented geodesic from $ x $ to $ y $.
If $ g \in \Gamma $ is hyperbolic,
write $ a(g) $, $ b(g) $ respectively,
for its attracting and repelling fixed points;
$ \alpha(g) = [b(g),a(g)] $
for its axis;
and $\lambda(g)$ for the oriented closed geodesic in $\Sigma$
corresponding to $g$, 
i.e., the image of $\alpha(g) \cap D$ in $\Sigma$.
If $ x,y \in C $ are distinct, then $ [x,y] \cap D $ projects to 
an oriented bi-infinite geodesic path, 
$ \lambda(x,y) $, in $ \Sigma $.
If $ x,y \in \Pi $, then this is a proper geodesic path, with a finite number,
$ \nu(x,y) $, of self-intersections.
Let $ \Delta = \{ (p,q) \in \Pi^2 \mid \nu(p,q) = 0 \} $, i.e.,
$\Delta$ consists of pairs $(p,q)$ of parabolic points such that
$ \lambda(p,q) $ is a proper geodesic arc.
(By an {\em arc}, we mean an embedded path.)
Given $ p \in \Pi $, write $ \Pi(p) = \{ q \in \Pi \mid (p,q) \in \Delta \} $.

Pick an element $ (p,q)\in \Delta $.
Then the proper arc $ \lambda(p,q) $ intersects a 
sufficiently small horocycle, $ \tau $, 
in precisely two points.
Let $ \tilde\tau \subseteq D $ be the horocircle centred at $ p $
which is a connected component of the inverse image of $ \tau $,
and let $ \{s_i\}_{ i\in \ZZ} $ be the inverse image of the two points 
in $ \tilde \tau $, located in this order, such that $ s_0=[p,q]\cap \tilde\tau $
and $ \theta(p)s_i=s_{i+2} $.
Then there is a unique element $ g(p,q)\in \Gamma $ 
such that $ g(p,q)p = q $ and $ g(p,q)^{-1}[p,q]\cap \tilde \tau = s_1$.
Namely, $ g(p,q)^{-1}[p,q] $ is the closure of the lift of $ \lambda(p,q) $
with endpoint $ p $ which is closest to $[p,q]$,
among the lifts of $ \lambda(p,q) $ with endpoint $ p $,
with respect to the preferred orientation of $ \tilde \tau $.
(See Figure \ref{Figure1}.)

In the quotient surface $ \Sigma $,
the oriented closed geodesic $ \lambda(g(p,q)) $ is homotopic to 
the simple oriented loop obtained by shortcutting the oriented arc
$ \lambda(p,q) $ by the horocyclic arc which is the image of 
the subarc of $ \tilde\tau $ bounded by $s_0$ and $s_1$.
Thus $ \lambda(g(p,q)) $ is a simple closed geodesic
disjoint from the proper geodesic arc $\lambda(p,q)$.
In particular, $[p, \theta(p)q]\cap \alpha(g(p,q))=\emptyset$.
In fact, the map
$ [(p,q) \mapsto g(p,q)] : \Delta \longrightarrow \Gamma $ 
is characterised by the following properties:
for all $ (p,q) \in \Delta $, we have $ g(p,q) p = q $,
$ g(q,p) g(p,q) = \theta(p) $, and $ [p, \theta(p) q] \cap \alpha(g(p,q)) =
\emptyset $.

\begin{figure}[h]
\begin{center}
\includegraphics{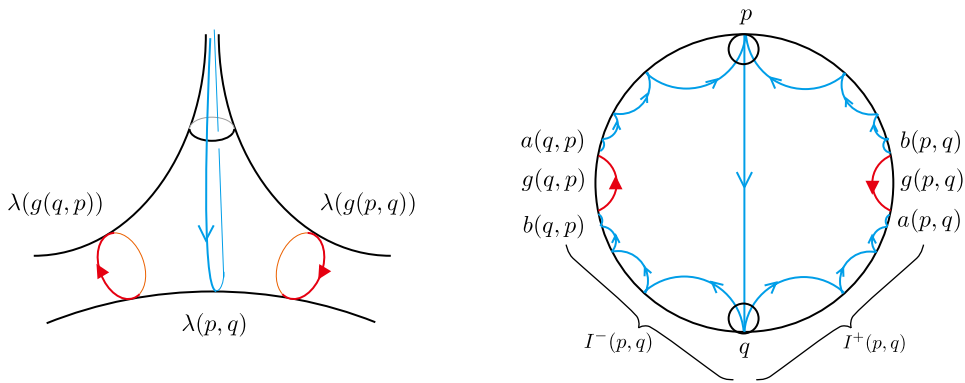}
\end{center}
\caption{
In the right figure, the two red arcs with thick arrows
represent the axes $\alpha(g(p,q))$ and $\alpha(g(q,p))$
of the hyperbolic transformations $g(p,q)$ and $g(q,p)$ respectively.
The blue arcs with thin arrows 
represent the oriented geodesic $[p,q]$
and its images by the infinite cyclic groups 
$ \langle g(p,q) \rangle $ and $ \langle g(q,p) \rangle $.
The three intersection points of the blue arcs and 
the horocircle $ \tilde \tau $ centred at $ p $ are 
$ s_{-1} $, $ s_0 $ and $ s_1 $, from left to right.
\label{Figure1}}
\end{figure}

Write $ a(p,q) = a(g(p,q)) $ and $ b(p,q) = b(g(p,q)) $.
Then the points $ p $, $ a(q,p) $, $ b(q,p) $, $ q $, $ a(p,q) $, $ b(p,q) $
occur in this order around $ C $.
Let $ I^+(p,q) = (q,a(p,q)) $, $ I^-(p,q) = (b(q,p),q) $
and $ I(p,q) = (b(q,p),a(p,q)) $ be open intervals in $ C $.
Thus $ I(p,q) = I^-(p,q) \cup \{ q \} \cup I^+(p,q) $,
$ I(p,q) \cap \theta(p)^n I(p,q) = \emptyset $ for all $ n \ne 0 $, and
$ I(p,q) \cap \theta(p)^n I(q,p) = \emptyset $ for all $ n $.

In the quotient surface $ \Sigma $,
the oriented simple closed geodesics 
$ \lambda(g(p,q)) $ and $ \lambda(g(q,p)) $
cut off a punctured annulus containing the geodesic arc $\lambda(p,q)$,
in which the simple geodesic rays 
$ \lambda(p, a(p,q)) $ and $ \lambda(p, b(q,p)) $
emanating from the puncture 
spiral to 
$ \lambda(g(p,q)) $ and $ \lambda(g(q,p)) $, respectively.
Thus, each of $ I^{\pm}(p,q) $ projects homeomorphically 
onto a \emph{gap} in the horocircle $\tau$,
in the sense of \cite[p.610]{Mcs}.
In fact, each of $ I^{\pm}(p,q) $ is a maximal connected subset of
$ C \bksl \{p\} $ consisting of points $x$ such that the geodesic ray
$\lambda(p,x)$ is non-simple.
Moreover, if $\lambda(p,x)$ is non-simple, then $x$ is 
contained in $ I^{\pm}(p,q) $ for some $ q\in \Pi(p) $
(see \cite{Mcs, TaWZ}).

Write $ {\script I}(p) = \{ I(p,q) \mid q \in \Pi(p) \} $. 
Then we obtain the following as a consequence of 
\cite[Corollary 5]{Mcs} and \cite{BiS} (see also \cite[Section 5]{TaWZ}):

\begin{thm}\label{C1}
The elements of $ {\script I}(p) $ are pairwise disjoint.
The complement, $ C \bksl \bigcup {\script I}(p) $, is a Cantor set
of Hausdorff dimension $ 0 $.
\end{thm}

Here, of course, the Hausdorff dimension is taken with respect
to the euclidean metric on the horocycle, $ \tau $.
Up to a scale factor, this is the same as the Euclidean metric
in the upper-half-space model with $ p $ at $ \infty $.
(Note that we could equally well use the circular metric on
the boundary, $ C $, induced by the Poincar\'e model, since
all the transition functions are M\"obius, and in particular, smooth.)

Write $ R(p) = \{ p \} \cup \Pi(p) \cup (C \bksl \bigcup {\script I}(p))
\subseteq C $.
This is a closed set, whose complementary components are precisely
the intervals $ I^\pm(p,q) $ for $ q \in \Pi(p) $.
Thus the set $R(p)$ is characterised by the following property:
a point $x\in C$ belongs to $R(p)$ if and only if $x\ne p$ and
the geodesic ray $\lambda(p,x)$ in $\Sigma$ is simple.

For $ p\in\Pi $, we define maps 
$ \vepsilon(p) $, $ \qq(p) $ and $ \gs(p) $
from $C \bksl R(p) $ to 
$\{  \mathord{+}, \mathord{-}\}$, $\Pi(p)$ and $\Gamma$, respectively, 
by the following rule.
If $ x \in C \bksl R(p) $, then $ x \in I^\epsilon(p,q) $ for some
unique $ \epsilon = \mathord{\pm} $ and $ q \in \Pi(p) $.
Define $ \vepsilon(p)(x) = \epsilon $, $ \qq(p)(x) = q $, and
$ \gs(p)(x) = g(p,q) $ or $ g(q,p)^{-1} $
according to whether $ \epsilon = \mathord{+} $ or $ \mathord{-}$.
Note that the definition is symmetric under simultaneously
reversing the orientation
on $ C $ and swapping $ \mathord{+} $ with $ \mathord{-} $.

It should be noted that if $ x \in C \bksl R(p) $,
then, in the quotient surface $ \Sigma $,
the geodesic ray $ \lambda(\qq(p)(x), x)=\lambda(q, x) $ 
is obtained from the non-simple geodesic ray $ \lambda(p,x) $
by cutting a loop, homotopic to $ \lambda(\gs(p)(x))=\lambda(g(p,q))$, 
and straightening the resulting piecewise geodesic
 (see Figure \ref{Figure2}).
(In the quotient, we are allowing ourselves to cut out any
peripheral loops that occur at the beginning.)
In particular, if $ x \in \Pi \bksl R(p) $, then 
both $ \lambda(p,x) $ and $ \lambda(\qq(p)(x), x)$
are proper geodesic paths in $ \Sigma $,
and their self-intersection numbers satisfy the inequality  
$ \nu(p,x) > \nu(\qq(p)(x), x) $.

\begin{figure}[h]
\begin{center}
\includegraphics{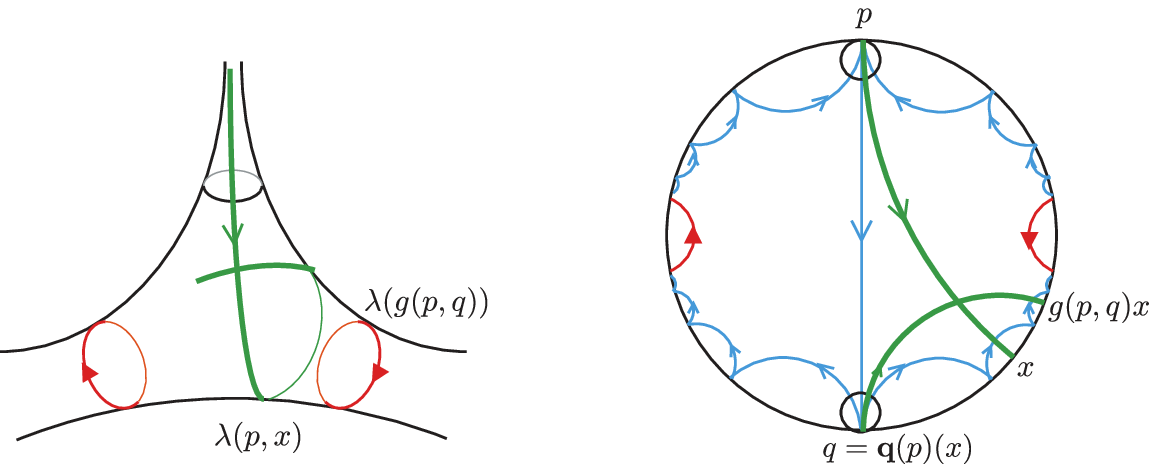}
\end{center}
\caption{
In the figure, we assume 
$ \vepsilon(p)(x) = \mathord{+} $ and so $ \gs(p)(x) = g(p,q) $.
\label{Figure2}}
\end{figure}

By repeatedly applying these maps, 
we associate for a given $ x \in C $, 
a sequence $ (g_i)_i $ in $ \Gamma $, $ (p_i)_i $ in $\Pi$, and
$ (\epsilon_i)_i $ in $\{  \mathord{+}, \mathord{-}\}$ as follows.

{\bf Step 0.}
Pick a parabolic point $ p\in\Pi $, and
define $p_0 =  p $.
Thus, $ p_0 $ is independent of $ x \in C $.

{\bf Step 1.}
If $ x \in R(p_0) $, we stop with the 1-element sequence $ p_0 $,
and  define $ (g_i)_i $ and $ (\epsilon_i)_i $ to be the empty sequence.
If $ x \notin R(p_0) $, set $ g_1 = \gs(p_0)(x) $, $ p_1 = g_1 p_0 $, 
$ \epsilon_1=\vepsilon(p_0)(x) $,
and continue to the next step.
(The sequences $ (g_i)_i $ and $ (\epsilon_i)_i $ begin with index
$ i = 1 $.)

{\bf Step 2.}
If $ x \in R(p_1) $, we stop with the 1-element sequences 
$ g_1 $ and $\epsilon_1$
and $2$-element sequence $p_0, p_1$.
If $ x \notin R(p_1) $, set $ g_2 = \gs(p_1)(x) $, 
$ p_2 = g_2 p_1 $ and $ \epsilon_2=\vepsilon(p_1)(x) $.

We continue this process, forever or until we stop.

We call the resulting sequences $ (g_i)_i $, $ (p_i)_i $ and $ (\epsilon_i)_i $ 
the \emph{derived sequences} for $x$.
More specifically, we call $ (g_i)_i $ and $ (p_i)_i $
the \emph{derived $ \Gamma $-sequence} and 
the \emph{derived $ \Pi $-sequence} for $x$,
respectively.

\begin{lem}\label{C2}
Let $ x\in C $, and let 
$ (g_i)_i $, $ (p_i)_i $ and $ (\epsilon_i)_i $ 
be the derived sequences for $x$.
Then the following hold.

{\rm (1)}
The sequences $ (p_i)_i $ and $ (\epsilon_i)_i $ are determined 
by the sequence $ (g_i)_i $ by the following rule:
$ p_i= h_i p_0 =h_i p $ where $ h_i = g_i g_{i-1} \cdots g_1 $,
and $ \epsilon_i= \mathord{+} $ or $ \mathord{-} $
according to whether $ g_i = g(p_{i-1},p_i) $ or $ g(p_{i-1},p_i)^{-1} $.

{\rm (2)}
A point $ y\in C $ has the derived $ \Gamma $-sequence 
beginning with $ g_1, g_2, \dots, g_n $ for some $ n\ge 1 $,
if and only if 
$ y \in \bigcap_{i=1}^n I^{\epsilon}(p_{i-1}, p_i) $.

{\rm (3)}
Set $ R = \bigcup_{p \in \Pi} R(p) $.
If $ x \notin R $, then the derived $ \Gamma $-sequence $ (g_i)_i $ is infinite.

{\rm (4)}
If $ x \in \Pi $, then the derived $ \Gamma $-sequence $ (g_i)_i $ is finite.
\end{lem}

\begin{proof}
(1), (2) and (3) follow directly from the definition of the derived sequences.
To prove (4),
let $ x $ be a point in $ \Pi $.
If $ x \in R(p) $, then $ (g_i)_i $ is the empty sequence.
So we may assume $ x \in \Pi \bksl R(p) $.
Then by repeatedly using the observation made
prior to the construction of the derived sequences,
we see that the self-intersection number  $ \nu(p_i,x) $
of the proper geodesic path $ \lambda(p_i,x) $ 
is strictly decreasing.
Hence $ \nu(p_n, x)=0 $ for some $n$.
This means that $x \in R(p_n)$ and so the derived sequences 
terminate at $n$.
\end{proof}

The following is an immediate consequence of Lemma \ref{C2}(2).

\begin{cor}\label{C3}
Suppose that $ x \in C $ has derived $ \Gamma $-sequence beginning with 
$ g_1 ,\ldots, g_n $ 
for some $ n \ge 1 $.
Then there is an open set, $ U \subseteq C $, containing $ x $, such that
if $ y \in U $, then $ g_1 ,\ldots, g_n $ is also an initial segment
of the derived $ \Gamma $-sequence for $ y $.
\end{cor}

Recall from Section \ref{SB} that $ A(\Gamma,\pph,p) $ denotes the
subgroup of $ \Aut(\Gamma) $ preserving $ \Pi $ setwise and
fixing $ p \in \Pi $.

\begin{lem}\label{C4}
Let $ \phi $ be an element of $ A= A(\Gamma,\pph,p) $
with $p=p_0$.
Then the following holds for every point $ x \in C $.
If $ (g_i)_i $, $ (p_i)_i $ and $ (\epsilon_i)_i $ are 
the derived sequences for $ x $,
then the derived sequences for $ \phi x $ are
$ (\phi(g_i))_i $, $ (\phi p_i)_i $ and $ (\deg(\phi)\epsilon_i)_i $.
\end{lem}

\begin{proof}
This can be proved through induction,
by using the fact that the following hold for each $ \phi \in A $.
\begin{enumerate}
\item
$ \phi(R(p)) = R(p) $.
\item
For any $ q\in \Pi(p) $, we have:
\begin{enumerate}
\item
If $ \phi $ is orientation-preserving, then $ \phi(\theta(p)) = \theta(p) $,
$ \phi(I^{\epsilon}(p,q)) = I^{\epsilon}(p,\phi(q)) $, 
$ \phi(g(p,q)) = g(p, \phi q) $, and $ \phi(g(q,p)) = g(\phi q, p) $.
\item
If $ \phi $ is orientation-reversing, then $ \phi(\theta(p)) = \theta(p)^{-1} $,
$ \phi(I^{\epsilon}(p,q)) = I^{-\epsilon}(p,\phi(q)) $, 
$ \phi(g(p,q)) = g(\phi q, p)^{-1} $, and
$ \phi(g(q,p)) = g(p, \phi q)^{-1} $.
\end{enumerate}
\end{enumerate}
\end{proof}

\section{Filling arcs}\label{SD}

Let $ x $ be a point in $ C $
and $ (p_i)_{i} $ the (finite or infinite)
derived $\Pi$-sequence for $x$.
Write $ \lambda_i = \lambda(p_{i-1},p_i) $ for the projection
of $ [p_{i-1},p_i] \cap D $ to $ \Sigma $.
This is a proper geodesic arc in $ \Sigma $.
We call the sequence $(\lambda_i)_i$ 
the \emph{derived sequence of arcs} for $x$.
We say that $ x $ is \emph{filling}
if the arcs $ (\lambda_i)_i $ 
eventually fill $ \Sigma $,
namely,  there is some $ n $ such that
$ \Sigma \bksl \bigcup_{i=1}^n \lambda_i $ is a union of open discs.
Let $ F $ be the subset of $C$ consisting of points which are filling.
In this section, we prove the following proposition.

\begin{prop}\label{D1}
The set $ F $ is open in $ C $, and its complement has 
Hausdorff dimension strictly less than $1$.
In particular, $ F $ has full measure.
\end{prop}

We begin with some preparation.
Let $ \gamma $ be a simple closed geodesic in $ \Sigma $,
and let $ X(\gamma) $ be the path-metric completion
of the component of $ \Sigma \bksl \gamma $ containing the cusp.
Then we can identify $ X(\gamma) $ as $ (H(G) \cap D)/G $, where
$ G = G(\gamma) $ is a subgroup of $ \Gamma $ containing 
$ Z= \stab_\Gamma(p) $,
and $ H(G) \subseteq D \cup C $ is the convex hull of the limit set
$ \Lambda G \subseteq C $.
In other words, 
$ X(\gamma) $
is the ``convex core'' of
the hyperbolic surface 
$ {\bbf H}^2/G $.
Note that 
$ G = G(\gamma) \cong \pi_1(X(\gamma)) $
and $ p \in \Lambda G $.

Let $ \delta $ be the closure of a component of $ \partial H(G) \cap D $.
This is a bi-infinite geodesic in $ D \cup C $.
Let $ J \subseteq C $ be the component of $ C \bksl \delta $ not
containing $ p $.
Thus, $ J $ is an open interval in $ C $, which is a component of
the discontinuity domain of $ G $.
Note in particular, that $ J \cap Gp = \emptyset $.

\begin{lem}\label{D2}
Suppose $ x \in J \bksl R(p) $, and
let $ g = \gs(p)(x) $, $ \epsilon = \vepsilon(p)(x) $ and
$ q = \qq(p)(x) $.
Then, if $ g \in G = G(\gamma) $, 
we have $ J \subseteq I^{\epsilon}(p,q) $.
In particular, $ \gs(p)(y)= g $ for every $ y\in J $.
\end{lem}

\begin{proof}
To simplify notation we can assume (via the orientation reversing
symmetry of the construction) that $ \epsilon = \mathord{+} $.
Note that $ q \in Gp \subseteq \Lambda G $, so $ [p,q] \subseteq H(G) $.
Also $ \alpha(g(p,q)) \subseteq H(G) $ and $ \delta \subseteq \partial H(G) $.
It follows that $ [p,q] $, $ \alpha(g(p,q)) $ and $ \delta $ are
pairwise disjoint.
Thus, $ J $ lies in a component of $ Y:= C \bksl \{p, q, a(p,q), b(p,q)\} $.
Since $ \epsilon = \mathord{+} $, the four points, 
$ p, q, a(p,q), b(p,q) $ are located in $ C $ in this cyclic order,
and so $ I^+(p,q)=(q, a(p,q)) $ is a component of $ Y $.
Since $ J $ and $ I^+(p,q) $ share the point $ x $,
we obtain the first assertion that $ J \subseteq I^{\epsilon}(p,q) $ 
with  $ \epsilon = \mathord{+} $.
The second assertion follows from the first assertion and  
the definition of $ \gs(p)(y) $.
\end{proof}

\begin{lem}\label{D3}
Suppose that $ x \in J $ and that the derived $ \Gamma $-sequence $ (g_i)_i $ 
for $ x $
is infinite.
Then there is some $ i $ such that $ g_i \notin G = G(\gamma) $.
\end{lem}

\begin{proof}
Suppose, for contradiction, that $ g_i \in G $ for all $ i $.
It follows that $ h_i=g_ig_{i-1}\cdots g_1 \in G $ for all $ i $, and so
$ p_i=h_i p \in Gp \subseteq \Lambda G $ for all $ i $.
By Lemma \ref{D2}, we have $ \gs(p)(y) = \gs(p)(x) = g_1 $
for all $ y\in J $.
(Here $ (p_i)_i $ is the derived $ \Pi $-sequence for $ x $ and $ p=p_0 $.)
Now, applying Lemma \ref{D2} with $ p_1 $ in place of $ p $, we get
that $ \gs(p_1)(y) = \gs(p_1)(x) = g_2 $.
Continuing inductively we get that $ \gs(p_i)(y) = g_{i+1} $ for all $ i $.
In other words, the derived $ \Gamma $-sequence for $ y $ is identical 
to that for $ x $,
and so, in particular, it must be infinite.
We now get a contradiction by applying Lemma \ref{C2}(4) to any
point $ y \in \Pi \cap J $.
\end{proof}

If we take $ B $ to be a standard horoball neighbourhood of 
the cusp,
then $ B \cap \gamma = \emptyset $ for 
all simple closed geodesic in $\Sigma$,
and so we
can identify 
$ B $
with a neighbourhood of the cusp in any $ X(\gamma) $.

\begin{lem}\label{D4}
There is some $ \theta < 1 $ such that for each simple closed geodesic,
$ \gamma $, the Hausdorff dimension of $ \Lambda G(\gamma) $ is at most
$ \theta $.
\end{lem}

\begin{proof}
This is an immediate consequence of \cite[Theorem 3.11]{FalM}
(see also \cite[Theorem 1]{Mat})
which refines the result of \cite{Tu1},
on observing that the groups $ G(\gamma) $ are uniformly ``geometrically tight'',
as defined in that paper.
Here, this amounts to saying that there is some fixed $ r \ge 0 $
(independent of $ \gamma $) such that the convex core,
$ X(\Gamma) $, is the union of $ B $ and the $ r $-neighbourhood of
the geodesic boundary of the convex core.
From the earlier discussion, we see that $ r $ is bounded above by the diameter
of $ \Sigma \bksl B $, and so in particular,
independent of $ \gamma $.
\end{proof}

Let $ L \subseteq S $ be the union of the limit sets $ \Lambda G $ as
$ G=G(\gamma) $ ranges over all subgroups of $\Gamma$
obtained from a simple closed geodesic $\gamma$ in $\Sigma$.
Applying Lemma \ref{D4}, we see that
$ L $ is a $ \Gamma $-invariant subset of $C$ of Hausdorff dimension
strictly less than 1.
This is because it is a countable union of the limit sets $ \Lambda G $
whose Hausdorff dimensions are uniformly bounded by a constant $ \theta<1 $.

Recall the set $ R = \bigcup_{p \in \Pi} R(p) $
defined in Lemma \ref{C2}(3).
Then $R$ is also $ \Gamma $-invariant 
and has Hausdorff dimension zero by Theorem \ref{C1}.

\begin{lem}\label{D5}
If $ x \in C \bksl (R \cup L) $, then $x$ is filling.
Namely, $ C \bksl (R \cup L) \subseteq F $.
\end{lem}

\begin{proof}
Suppose, for contradiction, that some $ x \in C \bksl (R \cup L) $
is not filling.
Then there must be some simple closed geodesic, $ \gamma $, in $ \Sigma $,
which is disjoint from every $ \lambda_i $,
where $(\lambda_i)_i$ is the derived sequence of arcs for the point $ x $.
Consider the hyperbolic surface $X(\gamma)$ 
and its fundamental group $ G = G(\gamma) \subseteq \Gamma $,
as described at the beginning of this section.
By hypothesis, $ x \notin \Lambda G $, and so $ x $ lies in some
component, $ J $, of the discontinuity domain of $ G $.
By Lemma \ref{D3}, there must be some $ i \in {\bbf N} $ with
$ g_i \notin G $.
Choose the minimal such $ i $.
Thus, $ h_{i-1} \in G $ but $ h_i \notin G $, where $ h_i=g_ig_{i-1}\cdots g_1 $.
We have $ p_{i-1}= h_{i-1}p \in \Pi \cap \Lambda G $ and 
$ p_i = h_ip \in \Pi \bksl \Lambda G $.
(The latter assertion can be seen as follows.
If $ p_i \in \Lambda G $ 
then $ p_i $ is a parabolic fixed point of $ G $.
Since $ X(\gamma) $ has a single cusp, 
there is an element $ f\in G $ such that $ p_i = f p_{i-1} $.
Since $ p_i= g_i p_{i-1} $, we have 
$ f^{-1}g_i \in \stab_{\Gamma}(p_{i-1})=\stab_G(p_{i-1}) $.
This implies $ g_i\in f G\subseteq G $, a contradiction.)
Therefore $ [p_{i-1},p_i] $ meets $ \partial H(G) $, giving the
contradiction that $ \lambda_i $ crosses $ \gamma $ in $ \Sigma $.
\end{proof}

\begin{proof}[Proof of Proposition \ref{D1}]
By Lemma \ref{D5}, we have $ C \bksl F \subseteq R \cup L $.
Since $ R $ and $ L $ both have Hausdorff dimension strictly less than $1$,
the same is true of $ C \bksl F $.
Thus, we have only to show that $ F $ is open.
Pick an element $ x\in L $. Then there is some $ n $ such that
$ \Sigma \bksl \bigcup_{i=1}^n \lambda_i $ is a union of open discs,
where $ (\lambda_i)_i $ is a derived sequence of arcs for $x$.
By Corollary \ref{C3}, there is an open neighbourhood $ U $ of $ x $ in $ C $
such that every $ y\in U $ shares the same initial
derived $\Gamma$-sequence
$g_1, \dots, g_n$ with $ x $.
Thus, every $ y\in U $ shares the same  beginning derived sequence of arcs
$(\lambda_i)_{i=1}^n $ with $ x $.
Hence every $ y\in U $ is filling, i.e., $U\subseteq F$.
\end{proof}

\section{Wandering}\label{SE}

Recall that $ \Map(\Sigma) $ is identified with $ M=A/Z $,
where $ A=A(\Gamma,\pph,p) $ and $ Z=Z(\Gamma, p) $, respectively,
are the stabilisers of $ p $ in
$ \Aut(\Gamma,\pph) $ and $ \Gamma $.
As described in Section \ref{SB},
$ A $ acts on $ C\bksl \{p\} $, and
$\Map(\Sigma)= M$ acts on the circle $ T=(C\bksl \{p\})/Z $.
The wandering domain $ W_M(T) $ is equal to $ W_A(C\bksl \{p\})/Z $,
because $ W_Z(C\bksl \{p\}) = C\bksl \{p\} $.
(See the general remark on the wandering domain 
given in Section \ref{SB}.)

Note that the set $ F $ in Proposition \ref{D1}
is actually an open set of $ C\bksl\{p\} $.
For this set $ F $, we prove the following lemma.

\begin{lem}\label{E1}
$ F \subseteq W_A(C\bksl\{p\}) $.
\end{lem}

\begin{proof}
We want to show that any $ x \in F $ has a wandering neighbourhood.
By assumption, some initial segment, $ \lambda_1 ,\ldots, \lambda_n $, of the
derived sequence of arcs for $ x $ fills $ \Sigma $.
By Corollary \ref{C3},
there is an open neighbourhood, $ U $, of $ x $,
such that for every $ y\in U $, the initial segment of length $n$ of
the derived sequence of arcs is identical with  
$ \lambda_1 ,\ldots, \lambda_n $.
Suppose that $ U \cap \phi U \ne \emptyset $
for some non-trivial element $ \phi $ of $ \Map(\Sigma)=A/Z $.
Pick a point $ y\in U \cap \phi U $ and set $ x = \phi^{-1} y \in U$.
By assumption, the derived sequences of arcs for both $ x $ and $ y $
begin with $ \lambda_1, \dots, \lambda_n $.
On the other hand, Lemma \ref{C4} implies that
the derived sequence of arcs for $ y=\phi x $
is equal to the image of that for $ x $ by $ \phi $.
Hence we see that 
$ \phi \lambda_i = \lambda_i $ for all $ i = 1 ,\ldots, n $.
It follows by Lemma \ref{E2} below, 
that $ \phi $ is the trivial element of $ \Map(\Sigma) $, a contradiction.
\end{proof}

In the above, we have used 
the following lemma which appears
to be well known, though we were unable to find an explicit reference.

\begin{lem}\label{E2}
Let $ \lambda_1 ,\ldots, \lambda_n $ be a set of proper oriented arcs in
$ \Sigma $ which together fill $ \Sigma $.
Suppose that $ \psi $ is a mapping class on $ \Sigma $ fixing the proper
homotopy class of each $ \lambda_i $.
Then $ \psi $ is trivial.
\end{lem}

\begin{proof}[Proof of Theorem \ref{B1}]
By Proposition \ref{D1}, 
$ F $ is an open set of $ C\bksl\{p\} $
whose complement has Hausdorff dimension strictly less than $ 1 $.
Since $ W_A(C\bksl\{p\}) $ contains $ F $ by Lemma \ref{E1},
its complement in $ C\bksl\{p\} $
also has Hausdorff dimension strictly less than $ 1 $.
Since $ W_M(T) = W_A(C\bksl\{p\})/Z $,
this implies that the non-wandering set, $ T\bksl W_M(T) $, 
has Hausdorff dimension strictly less than $ 1 $.
\end{proof}

\begin{proof}[Proof of Lemma \ref{E2}]
Fix any complete finite-area hyperbolic structure on $ \Sigma $,
and use it to identify $ {\tilde \Sigma} $ with $ {\bbf H}^2 $.
Construct a graph, $ M $, as follows.
The vertex set, $ V(M) $, is the set
of bi-infinite geodesics which
are lifts of the arcs $ \lambda_i $ for all $ i $.
Two arcs $ \mu, \mu' \in V(M) $ are deemed adjacent in $ M $ if either
(1) they cross (that is, meet in $ {\bbf H}^2 $), or
(2) they have a common ideal point in $ \partial {\bbf H}^2 $, and
there is no other arc in $ V(M) $ which separates $ \mu $ and $ \mu' $.
One readily checks that
$ M $ is locally finite.
Moreover, the statement that the arcs 
$ \lambda_i $
fill $ \Sigma $ is equivalent to the statement that $ M $ is connected.
Note that $ \Gamma = \pi_1(\Sigma) $ acts on $ M $ with finite quotient.
Note also that $ M $ can be defined formally in terms of ordered
pairs of points in $ S^1 \equiv \partial {\bbf H}^2 $
(that is corresponding to the endpoints of the geodesics, and where
crossing is interpreted as linking of pairs).
The action of $ \Gamma $ on $ M $ is then induced by the dynamically
defined action of $ \Gamma $ on $ S^1 $.

Now suppose that $ \psi \in \Map(\Sigma) $.
Lifting some representative of $ \psi $ and extending to the ideal circle
gives us a homomorphism of $ S^1 $, equivariant via the corresponding
automorphism of $ \Gamma $.
Suppose that $ \psi $ preserves each arc $ \lambda_i $, as in the
hypotheses.
Then $ \psi $ induces an automorphism,
$ f : M \longrightarrow M $.
Given some 
$ \mu \in V(M) $,
by choosing a suitable lift of $ \psi $, we
can assume that $ f(\mu) = \mu $.

We claim that this implies that $ f $ is the identity on $ M $.
To see this, first let $ V_0 \subseteq V(M) $ be the set of vertices
adjacent to $ \mu $.
This is permuted by $ f $.
Consider the order on $ V_0 $ defined as follows.
Let $ I_R $ and $ I_L $, respectively, be the closed intervals of $ S^1 $
bounded by $ \partial \mu $
which lies to the right and left of $ \mu $.
Orient each of $ I_R $ and $ I_L $ so that 
the initial/terminal points of $ \mu $, respectively, are those of 
the oriented $ I_R $ and $ I_L $.
Each $ \nu \in V_0 $ determines a unique pair 
$ (x_R(\nu),x_L(\nu)) \in I_R\times I_L $ such that
$ x_R(\nu) $ and $ x_L(\nu)$ are the endpoints of $ \nu $.
Now we define the order $ \le $ on $ V_0 $,
by declaring that $ \nu \le \nu' $ if either (i) $ x_R(\nu) \le x_R(\nu') $ or
(ii) $ x_R(\nu) = x_R(\nu') $ and $  x_L(\nu) \le x_L(\nu') $.
This order must be respected by $ f $,
because $ f $ preserves the orders on $ I_R $ and $ I_L $.
Since $ V_0 $ is finite,
we see that $ f|V_0 $ is the identity.
The claim now follows by induction, given that $ M $ is connected.

It now follows that the lift of $ \psi $ is the identity on the set
of all endpoints of elements of $ V(M) $.
Since this set is dense in $ S^1 $, it follows that it is the identity
on $ S^1 $, and we deduce that $ \psi $ is the trivial mapping class
as required.
\end{proof}

\end{document}